\documentclass[12pt, oneside]{article}

\usepackage{amsmath}
\usepackage{amsfonts}
\usepackage{amssymb}
\usepackage{amsthm}
\usepackage{mathtools}
\usepackage{enumerate}
\usepackage{array, booktabs}
\usepackage{bm}
\usepackage{bbm}
\usepackage{mathrsfs}
\usepackage{authblk} %
\usepackage{mathptmx} %
\usepackage{stmaryrd} %
\usepackage{filecontents}
\usepackage[style=numeric, backend=biber, doi=false, url=false, isbn=false,
uniquename=false, giveninits]{biblatex}
\renewcommand*{\bibfont}{\small}

\DeclareMathAlphabet{\mathcal}{OMS}{cmsy}{m}{n} %

\allowdisplaybreaks

\newtheorem{theorem}{Theorem}[section]
\newtheorem{lemma}{Lemma}[section]
\newtheorem{proposition}{Proposition}[section]
\newtheorem{corollary}{Corollary}[section]
\theoremstyle{definition}
\newtheorem{definition}{Definition}[section]
\theoremstyle{remark}
\newtheorem{remark}{Remark}[section]
\numberwithin{equation}{section}
\newcommand{\largelem}{L{\footnotesize EMMA} }
\newcommand{\largethm}{T{\footnotesize HEOREM} }

\def \A{\mathcal{A}}
\def \B{\mathcal{B}}
\def \C{\mathbbm{C}}
\def \F{\mathcal{F}}
\def \G{\mathcal{G}}
\def \H{\mathcal{H}}
\def \L{\mathcal{L}}
\def \N{\mathbbm{N}}
\def \Q{\mathbbm{Q}}
\def \R{\mathbbm{R}}
\def \S{\mathcal{S}}
\def \Z{\mathbbm{Z}}

\def \defby{\vcentcolon =}

\newcommand\given{\nonscript\:\delimsize\vert\nonscript\:\mathopen{}}

\DeclarePairedDelimiterX\set[1]\{\}{%
	#1
}

\DeclarePairedDelimiter{\norm}{\lVert}{\rVert}
\DeclarePairedDelimiter{\abs}{\lvert}{\rvert}
\DeclarePairedDelimiter{\vertiii}{\lvert\kern-0.25ex\lvert\kern-0.25ex\lvert}
{\rvert\kern-0.25ex\rvert\kern-0.25ex\rvert}
\DeclarePairedDelimiterX{\inp}[2]{\langle}{\rangle}{#1, #2}
\newcommand{\conj}[1]{\overline{#1}}
\def \supp{\operatorname{supp}}
\newcommand{\Supp}{\operatorname{Supp}}

\renewcommand{\epsilon}{\varepsilon}

\renewcommand{\Re}{\operatorname{Re}}

\newcommand{\indic}{\mathbbm{1}}
\newcommand{\torus}{\mathbbm{T}^2}
\newcommand{\wickproduct}{:\!\!\! \times \!\!\!: \,}
\newcommand{\wick}[1]{\vcentcolon\kern-0.4ex#1\kern-0.5ex\vcentcolon}
\newcommand{\rconst}{\mathcal{R}}
\newcommand{\conv}{\star}
\newcommand{\kernel}{\mathcal{K}}
\newcommand{\green}{\mathcal{G}}

\newcommand{\contisp}{\mathcal{C}}
\newcommand{\besov}{\mathcal{B}}
\newcommand{\zset}{\mathscr{Z}}

\newcommand{\sobolev}{\mathcal{H}}
\newcommand{\paral}{\varolessthan}
\newcommand{\parag}{\varogreaterthan}
\newcommand{\resonant}{\varodot}
\def \P{\mathbbm{P}}
\def \expect{\mathbbm{E}}
\newcommand{\Law}{\operatorname{Law}}

\newcommand*{\TitleFont}{%
      \usefont{\encodingdefault}{\rmdefault}{b}{n}%
      \fontsize{16}{20}%
      \selectfont}
\title{\TitleFont Characterization of the support for Wick powers of the additive stochastic heat equation}
\author{Toyomu Matsuda}
\affil{{\small Graduate School of Mathematics, Kyushu University, \\
744 Motooka, Nishi-ku, Fukuoka 819-0395, Japan \\
email address: ma218004@math.kyushu-u.ac.jp}}
\date{}

\begin{document}
\maketitle

\begin{abstract}
Let $Z$ be the stationary solution of the additive stochastic heat equation
\begin{equation*}
  \partial_t Z = (\Delta - 1) Z + \xi \quad \mbox{on } \torus = \R^2 / \Z^2,
\end{equation*}
where $\xi$ is the space-time white noise.
The aim of this paper is to determine the support of Wick powers $\{Z^{:k:}\}_{k=1}^{\infty}$.
This leads to an elementary proof of a support theorem for the dynamic $P(\Phi)_2$ equation.
In addition, we show that the approach can be used to determine the support of the law of the Gaussian multiplicative chaos
in the $L^2$-phase.
\vspace{0.5em}\\
\emph{Key words}: Support theorem, Additive stochastic heat equation, Wick powers, Gaussian multiplicative chaos
\end{abstract}

\section{Introduction}
Characterizing the support of the law of a random variable is a fundamental problem in probability theory.
In stochastic analysis, this characterization is often called a ``support theorem", which dates back to a seminal work
by Stroock and Varadhan \cite{SV72} in the context of stochastic differential equations (SDEs).
Since then, a large amount of literature has been devoted to support theorems in various contexts,
including stochastic partial differential equations (SPDEs). The most relevant to this paper is the work by
Ledoux, Qian and Zhang \cite{LQZ02}. They used Lyons' rough path theory \cite{Lyons98} to observe that
the support of the law of an SDE can be specified by characterizing the support of the enhanced noise.

Another interesting topic in stochastic analysis is research on singular SPDEs. One of the simplest example of singular SPDEs,
which we deal with in this paper, is the dynamic $P(\Phi)_2$ equation
\begin{equation}\label{eq:P_Phi}
  \partial_t \Phi = (\Delta - 1) \Phi - \sum_{k=0}^N a_k \Phi^{k} + \xi \quad \mbox{on } \torus = \R^2 / \Z^2,
\end{equation}
where $N$ is odd, $a_N > 0$ and $\xi$ is the space-time white noise on $\R \times \torus$.
This equation was first solved by Da Prato and Debussche \cite{DD03}.
Mourrat and Weber \cite{MW2dim} and Tsatsoulis and Weber \cite{TW18} brought a rough path point of view
to the $P(\Phi)_2$ equation in the course of recent developments of singular SPDEs initiated by
Hairer \cite{Hai14} and Gubinelli, Imkeller and Perkowski \cite{GIP15}.

Since recent theories of singular SPDEs are inspired by rough path theory, it is natural to expect that
support theorems extends to the setting of singular SPDEs in the spirit of \cite{LQZ02}. This direction of research was first
carried out by Chouk and Friz \cite{CF18}. They succeeded in characterizing the support of the law of the generalized
parabolic Anderson equation in two dimensions. Their idea was imported by Tsatsoulis and Weber \cite{TW18} to
the support description of the dynamic $\Phi^4_2$ equation, which then leads to a proof of exponential ergodicity of
the dynamics. However, extension to the general $P(\Phi)_2$ equation was hampered by a technical difficulty which
will be explained now.

According to \cite{TW18}, the solution of the dynamic $P(\Phi)_2$ equation is a continuous function of
$(Z^{:k:})_{k=1}^N$, where $Z$ is a solution of the additive stochastic heat equation
\begin{equation*}
  \partial_t Z = (\Delta - 1) Z + \xi
\end{equation*}
and $Z^{:k:}$ is the $k$ th Wick power of $Z$. In view of \cite{LQZ02}, the support theorem for the dynamic
$P(\Phi)_2$ equation reduces to characterizing the support of the law of $(Z^{:k:})_{k=1}^N$.
In \cite{TW18}, the authors show that the support of the law of $(Z^{:k:})_{k=1}^3$ is the closure of
\begin{equation*}
  \set{(H_k(h, R))_{k=1}^3 \given h \in \H \mbox{ and } R \geq 0},
\end{equation*}
where $H_k$ is a $k$ th Hermite polynomial and $\H$ is the Cameron-Martin space of $Z$.
That the former is included in the latter is an easy consequence of smooth approximations of the noise $\xi$, which
corresponds to a Wong-Zakai approximation in the context of SDEs. This inclusion obviously extends to the case
of general $N$.
The difficult part is to prove the other inclusion.
The crucial step, motivated by \cite{CF18}, is to construct a sequence $\{h_n\}_{n=1}^{\infty}$
of smooth functions such that, if we set
\begin{equation*}
  T_{-h_n - Z_n} Z^{:k:} \defby \sum_{l=0}^k \binom{k}{l} Z^{:k:} (-h_n - Z_n)^{k-l},
\end{equation*}
where $Z_n$ is a smooth approximation of $Z$, the shifted driver $(T_{-h_n-Z_n} Z^{:k:})_{k=1}^3$
converges to $(H_k(0, R))_{k=1}^3$.
When $N=3$, they constructed $\{h_n\}$ explicitly. However, explicit construction of $\{h_n\}$ for general $N$ seems impossible.
This prevents them from determining the support of the law of $(Z^{:k:})_{k=1}^{N}$ for general $N$.

Surprisingly, Hairer and Schﾃｶnbauer \cite{HS19} proved support theorems in a general framework of singular SPDEs.
In particular, they proved support theorems for the dynamic $P(\Phi)_2$ equation and for the dynamic
$\Phi^4_3$ equation, the three dimensional version of the dynamic $\Phi^4_2$ equation.
However, their work is based on state-of-the-art theory of regularity structures, which seems too heavy
for the $P(\Phi)_2$ equation. In addition, their work does not directly determine the support of
$(Z^{:k:})_{k=1}^{\infty}$, which is of independent interest.

The aim of this paper is to complete the program of \cite{TW18} by characterizing the support of $(Z^{:k:})_{k=1}^{\infty}$.
Namely, we prove that the support of the law of $(Z^{:k:})_{k=1}^{\infty}$ is the closure of
\begin{equation*}
  \set{ (H_k(h, R))_{k=1}^{\infty} \given h \in \H \mbox{ and } R \geq 0}.
\end{equation*}
See Theorem \ref{thm:main}.
Our approach only uses elementary tools from Besov space theory.

As the distribution of $Z(t)$ is a massive Gaussian free field, the support theorem for $Z$
offers similar results for the Gaussian free field. In the final part of this paper, we go further by giving a
support theorem for a Gaussian multiplicative chaos, a random measure formally viewed as an exponential of a Gaussian free field.
Indeed, we prove that the law of the Gaussian multiplicative chaos has the full support in the space of measures on $\torus$
 (Theorem \ref{thm:support_of_gmc}).
However, we prove this under the assumption that the Gaussian multiplicative chaos has the second moment.

\subsection{Outline}
In Section \ref{sec:preliminaries}, we list some definitions and results from Besov space theory, Hermite polynomials
and the additive stochastic heat equation.
Section \ref{sec:support} is the main part of this paper. It begins with the statement of the main theorem
(Theorem \ref{thm:main}). In \ref{subsec:technical} we prove some technical lemmas and in \ref{subsec:proof_of_main_thm}
we proceed to the proof of the main theorem.
In \ref{subsec:complex} we briefly discuss its complex counterpart.
In Section \ref{sec:gmc}, we move our attention to Gaussian multiplicative chaos and characterize the support
of its law (Theorem \ref{thm:support_of_gmc}). In Appendix \ref{sec:estimate_of_she}, we provide technical estimates of the additive stochastic heat equation
which are used in Section \ref{sec:support}.
\subsection{Notations}\label{subsec:notations}
\begin{enumerate}[(i)]
  \item $\N \defby \{0, 1, 2, \ldots\}$ and $\R_+ \defby [0, \infty)$.
  \item  $(2k-1)!! \defby (2k-1) (2k-3) \cdots 1$ and $(2k)!! \defby (2k) (2k-2) \cdots 2$.
  \item $\torus \defby \R^2 / \Z^2$ is a two-dimensional torus. We set
  $e_m(x) \defby e^{2 \pi i m \cdot x}$ for $m \in \Z^2$ and $x \in \torus$.
  \item We set
  \begin{equation*}
    \Pi_n (\sum_{m \in \Z^2} a_m e_m ) \defby \sum_{\abs{m} \leq n} a_m e_m, \qquad (a_m \in \C).
  \end{equation*}
  \item We set
  \begin{equation*}
    \inp{f}{g} \defby \int_{\torus} f(x) \conj{g(x)} dx,
  \end{equation*}
  and
  \begin{equation*}
    \F f(m) \defby \hat{f}(m) \defby \int_{\torus} f(x) e^{-2 \pi i m \cdot x} dx.
  \end{equation*}
  \item For a random variable $X$, we denote by $\Law(X)$ the probability measure given by
  $\Law(X)(A) \defby \P(X \in A)$.
  \item Given a measure $\mu$ on a topological space $E$, we denote by $\Supp(\mu)$ the set
  \begin{equation*}
    \set{x \in E \given \mu(U) > 0 \mbox{ for every neighborhood } U \mbox{ of } x}.
  \end{equation*}
  \item $\contisp^{\gamma} = \B^{\gamma}_{\infty, \infty}$ is the Besov space of regularity $\gamma$ on $\torus$.
  See Section \ref{subsec:besov}.
  \item $H_k$ is a $k$-th Hermite polynomial. See Section \ref{subsec:hermite}.
  \item We write $A \lesssim B$ if there exists a constant $C \in (0, \infty)$ such that $A \leq C B$.
  When we emphasize that $C$ depends on parameters $a, b, \ldots$, we write $A \lesssim_{a, b, \ldots} B$.
\end{enumerate}

\section{Preliminaries}\label{sec:preliminaries}
\subsection{Besov spaces}\label{subsec:besov}
We refer the reader to \cite{Bahouri2011} for thorough treatment of Besov space theory.
Presented here is a minimum of Besov space theory which will be used.

We fix smooth, radial functions
$\chi_{-1}$, $\chi :\R^2 \to [0, 1]$ which satisfy
\begin{align*}
  &\supp(\chi_{-1}) \subset B(0, 4/3), \\
  &\supp(\chi) \subset B(0, 8/3) \setminus B(0, 3/4), \\
  &\chi_{-1} + \sum_{k = 0}^{\infty} \chi(\cdot/2^k) = 1.
\end{align*}
For $k \in \N$, we set $\chi_{k} \defby \chi(\cdot/2^k)$ and let $\eta_k$ be the inverse Fourier transform of $\chi_k$.
For a distribution $f$ on the torus $\torus$, we set
\begin{equation*}
  \Delta_k(f)(x) \defby f\Big(\sum_{m \in \Z^2} \eta_k(x + m - \cdot)\Big).
\end{equation*}
\begin{definition}
  For $p, q \in [1, \infty]$ and $\alpha \in \R$, we define
  \begin{equation*}
    \norm{f}_{\B^{\alpha}_{p, q}} \defby \norm{(2^{\alpha k} \norm{\Delta_k f}_{L^p(\torus)})_{k=-1}^{\infty}}_{l^q}.
  \end{equation*}
  The Besov space $\B^{\alpha}_{p, q}$ is the completion of $C^{\infty}(\torus)$ under $\norm{\cdot}_{\besov_{p,q}^{\alpha}}$.
  In particular, we set $\contisp^{\alpha} \defby \besov_{\infty, \infty}^{\alpha}$.
\end{definition}
\begin{definition}
  For $f, g \in C^{\infty}(\torus)$ we define the paraproduct
  \begin{equation*}
    f \paral g \defby g \parag f \defby \sum_{j < k-1} \Delta_j f \Delta_k g
  \end{equation*}
  and the resonance term
  \begin{equation*}
    f \resonant g \defby \sum_{\abs{j-k} \leq 1} \Delta_j f \Delta_k g.
  \end{equation*}
  We have the Bony decomposition
    $fg = f \paral g + f \resonant g + f \parag g$.
\end{definition}
In this paper, we need the following two results.
\begin{proposition}\label{prop:paraproduct_estimates}\leavevmode
  \begin{enumerate}[(i)]
    \item Let $\alpha_1, \alpha_2 \in \R$ and set $\alpha \defby \min\{\alpha_1, 0\} + \alpha_2$.
    Then we have
    \begin{equation*}
    \norm{f \paral g}_{\contisp^{\alpha}} \lesssim_{\alpha_1, \alpha_2}  (\norm{f}_{L^{\infty}} \indic_{\{\alpha_1 \geq 0\}} +
    \norm{f}_{\contisp^{\alpha_1}} \indic_{\{\alpha_1 < 0\}}) \norm{g}_{\contisp^{\alpha_2}}.
    \end{equation*}
    In particular, $f \paral g$ is well-defined for every pair of distributions $(f, g)$.
    \item Let $\alpha_1, \alpha_2 \in \R$ such that $\alpha \defby \alpha_1 + \alpha_2 > 0$.
    Then we have
    \begin{equation*}
      \norm{f \resonant g}_{\contisp^{\alpha}} \lesssim_{\alpha_1, \alpha_2}
      \norm{f}_{\contisp^{\alpha_1}} \norm{g}_{\contisp^{\alpha_2}}.
    \end{equation*}
    In particular, $f \resonant g$ is well-defined for $f \in \contisp^{\alpha_1}$ and
    $g \in \contisp^{\alpha_2}$ with $\alpha_1 + \alpha_2 > 0$.
  \end{enumerate}
\end{proposition}
\begin{proof}
  See \cite[\largelem 2.1]{GIP15}.
\end{proof}
\begin{proposition}\label{prop:scaling_in_besov}
  Set $\Lambda_{\lambda} f \defby f(\lambda \cdot)$.
  \begin{enumerate}[(i)]
    \item We have
    \begin{equation*}
      \norm{\Lambda_{\lambda} f}_{\contisp^{\alpha}} \lesssim_{\alpha} \lambda^{\alpha} \norm{f}_{\contisp^{\alpha}}
    \end{equation*}
    for every $\alpha \in (0, \infty)$, $\lambda \in \N \setminus \{0\}$ and $f \in \contisp^{\alpha}$.
    \item Assume $\F f(0) = 0$. Then we have
    \begin{equation*}
      \norm{\Lambda_{\lambda} f}_{\contisp^{\alpha}} \lesssim_{\alpha} \lambda^{\alpha} \norm{f}_{\contisp^{\alpha}}
    \end{equation*}
    for every $\alpha \in \R \setminus\{0\}$, $\lambda \in \N \setminus \{0\}$ and $f \in \contisp^{\alpha}$.
  \end{enumerate}
\end{proposition}
\begin{proof}
  See \cite[\largelem A.4]{GIP15}.
\end{proof}
\subsection{Hermite polynomials}\label{subsec:hermite}
We refer the reader to \cite{Jan97} for more about Hermite polynomials.
\begin{definition}\label{def:hermite_polynomials}
  We define Hermite polynomials $\{H_k(x, C)\}_{k \in \N}$ by the identity
  \begin{equation}\label{eq:hermite_identity}
    e^{t x - 2^{-1} C t^2} = \sum_{k=0}^{\infty} \frac{t^k}{k!} H_{k}(x, C).
  \end{equation}
  We set $H_k(x) \defby H_k(x, 1)$.
\end{definition}
\begin{remark}\label{remark:explicit_form_of_hermite}
  We have the explicit representation of $H_k$ given by
  \begin{equation*}
    H_k(x, C) = k! \sum_{l=0}^{2^{-1}k} \frac{(-1)^l C^l}{2^l l! (k-2l)!} x^{k-2l}.
  \end{equation*}
\end{remark}
In this paper, we need the following two elementary identities.
\begin{proposition}\label{prop:hermite_binomial_expansion}
  We have the identity
  \begin{equation*}
    H_k(x + y, C) = \sum_{l=0}^k \binom{k}{l} H_l(x, C) y^{k-l}.
  \end{equation*}
\end{proposition}
\begin{proof}
  By \eqref{eq:hermite_identity}, we have
  \begin{equation*}
    \sum_{k=0}^{\infty} \frac{t^k}{k!} H_k(x + y, C) = e^{ty} e^{tx - \frac{Ct^2}{2}}
    = \sum_{k, l=0}^{\infty} \frac{t^{k+l}}{k!l!} H_k(x, C) y^l. \qedhere
  \end{equation*}
\end{proof}
\begin{proposition}\label{prop:hermite_inversion}
  We have the identity
  \begin{equation*}
    x^k = k! \sum_{l=0}^{2^{-1}k} \frac{C^l H_{k-2l}(x, C)}{2^l l! (k-2l)!}.
  \end{equation*}
\end{proposition}
\begin{proof}
  By \eqref{eq:hermite_identity}, we have
  \begin{equation*}
    \sum_{k=0}^{\infty} \frac{t^k}{k!} x^k = e^{\frac{Ct^2}{2}} \sum_{k=0}^{\infty} \frac{t^k}{k!} H_k(x, C)
    = \sum_{k, l=0}^{\infty} \frac{t^{k+2l}}{k!l!} \left(\frac{C}{2}\right)^l H_k(x, C). \qedhere
  \end{equation*}
\end{proof}

\subsection{Additive stochastic heat equation}\label{subsec:she}
Let $\alpha \in (0, \frac{1}{2})$.
The rigorous definitions of the space-time white noise $\xi$ and
the stationary solution of the additive stochastic heat equation $Z$ are given by
the following;
\begin{definition}\label{def:def_of_white_noise_etc}\leavevmode
  \begin{enumerate}[(i)]
    \item The space-time white noise $\xi$ is a centered Gaussian family
    $\set{\xi(u) \given u \in L^2(\R \times \torus)}$ with the property
    \begin{equation*}
      \expect[\xi(u) \xi(v)] = \int_{\R \times \torus} u(t, x) v(t, x) dt dx.
    \end{equation*}
    \item The stationary solution $Z$ of the additive stochastic heat equation is a
    $C(\R_+;$ \\ $\contisp^{-\alpha})$-valued
    random variable such that for every $t \in \R_+$ and $\phi \in C^{\infty}(\torus)$,
    \begin{equation*}
      \inp{Z(t)}{\phi} = \int_{\R \times \torus} \inp{\kernel(t-s, \cdot - y)}{\phi} \xi(dsdy)
    \end{equation*}
    almost surely, where the integral is in the sense of It\^o-Wiener integral and
    \begin{equation*}
      \kernel(t, x) \defby \indic_{\{t \geq 0\}} \sum_{m \in \Z^2} e^{-(1+4\pi^2 \abs{m}^2)t} e_m(x).
    \end{equation*}
    \item We denote by $Z^{:k:}$ a $C(\R_+; \contisp^{-\alpha})$-valued random variable such that for every
    $t \in \R_+$ and $\phi \in C^{\infty}(\torus)$,
    \begin{equation*}
      \inp{Z^{:k:}(t)}{\phi} = \int_{(\R \times \torus)^k} \inp{\prod_{j=1}^k \kernel(t-s_j, \cdot - y_j)}{\phi}
      \prod_{j=1}^k \xi(ds_j dy_j),
    \end{equation*}
    where the integral is in the sense of multiple It\^o-Wiener integral.
  \end{enumerate}
\end{definition}
\begin{remark}\leavevmode
  \begin{enumerate}[(i)]
    \item We refer the reader to \cite[Chapter 7]{Jan97} for stochastic integration with respect to the white noise.
    \item The existence of $Z^{:k:}$ is a consequence of Besov space version of Kolmogorov continuity theorem.
    See \cite[\largelem 9 and \largelem 10]{MW2dim} and \cite[Section 2.1]{TW18}.
    \item Similarly, $\xi$ has a distribution-valued modification.
    \item $Z^{:k:}$ is called $k$ th Wick power of $Z$.
  \end{enumerate}
\end{remark}
\begin{remark}\label{rem:bm_from_white_noise}
  Let $W_m$ be a continuous modification of
  \begin{equation*}
    t \mapsto \int_{\R \times \torus} \indic_{[\min\{0, t\}, \max\{0, t\}]}(s) e_m(y) \xi(dsdy).
  \end{equation*}
  Then $\{W_m\}_{m \in \Z^2}$ are complex Brownian motions and we have
  \begin{equation*}
    \inp{Z(t)}{e_m} = \int_{-\infty}^t e^{-(1+4\pi^2 \abs{m}^2)(t-s)} dW_m(s) \quad \mbox{almost surely}.
  \end{equation*}
\end{remark}
We set $\F^0 \defby \sigma(\xi(u) \vert u \in L^2(\R \times \torus))$ and
denote by $\F$ the usual augmentation of $\F^0$. We set
\begin{equation*}
  H^{:k:} \defby \overline{\set{H_k(\xi(u)) \given \norm{u}_{L^2(\R \times \torus)} = 1}}^{L^2(\Omega, \F, \P)}.
\end{equation*}
Then we have the Wiener chaos decomposition $L^2(\Omega, \F, \P) = \oplus_{k=0}^{\infty} H^{:k:}$.
We denote by $\zeta_1 \wickproduct \zeta_2$ the Wick product of $\zeta_1$ and $\zeta_2$.
See \cite[Chapter 2 and 3]{Jan97} for more detail.

Let $Z_n \defby \Pi_n Z$. See \ref{subsec:notations} for the definition of $\Pi_n$. Then we have
\begin{equation*}
  Z_n(t, x) = \int_{\R \times \torus} \kernel_n(t-s, x-y) \xi(ds dy)
\end{equation*}
almost surely, where $\kernel_n(t) \defby \Pi_n \kernel(t)$.
We define
\begin{align*}
  Z_n^{:k:}(t, x) &\defby \overbrace{Z_n(t, x) \wickproduct \cdots \wickproduct Z_n(t, x)}^{k\mbox{ times}}\\
  &= \int_{(\R \times \torus)^k} \prod_{j=1}^k \kernel_n(t-s_j, x-y_j) \prod_{j=1}^k \xi(ds_j dy_j).
\end{align*}
The second equality follows from the multiplication formula(\cite[\largethm 7.33]{Jan97}).
In addition, if we set $\rconst_n \defby \expect[ Z_n(t, x)^2 ]$, we have
\begin{equation*}
  Z_n^{:k:}(t, x) = H_k(Z_n(t, x), \rconst_n)
\end{equation*}
by \cite[\largethm 3.19]{Jan97}.
As shown in Corollary \ref{cor:fin_dim_approxim_of_she}, we have
\begin{equation*}
  \lim_{n \to \infty} \expect[ \sup_{0 \leq t \leq T} \norm{Z^{:k:}(t) - Z_n^{:k:}(t)}_{\contisp^{-\alpha}}^2] = 0.
\end{equation*}
\begin{remark}\label{rem:space_time_approximation_of_she}
  We similarly have space-time approximations.
  Let $\rho \in \S(\R \times \R^2)$ with $\int_{\R \times \R^2} \rho(t, x) dt dx = 1$ and let
  \begin{equation*}
    \rho_n(t, x) \defby n^4 \rho(n^2t, nx) \quad \mbox{and} \quad
    \tilde{Z}_n(t, x) \defby \int_{\R \times \torus} [\kernel * \rho_n](t-s, x-y) \xi(dsdy).
  \end{equation*}
  Furthermore, we set $\tilde{\rconst}_n \defby \expect[Z_n(t, x)^2]$ and
  $\tilde{Z}_n^{:k:}(t,x) \defby H_k(\tilde{Z}_n(t, x), \tilde{\rconst}_n)$.
  Then, we have
  \begin{equation*}
    \lim_{n \to \infty} \expect[ \sup_{0 \leq t \leq T} \norm{Z^{:k:}(t) - \tilde{Z}_n^{:k:}(t)}_{\contisp^{-\alpha}}^2 ] = 0.
  \end{equation*}
\end{remark}

\section{Support theorem: Wick powers of the stochastic heat equation}\label{sec:support}
Let $Z$ be the stationary solution of the additive stochastic heat equation on the torus $\torus$.
See Definition \ref{def:def_of_white_noise_etc}.
We fix $\alpha \in (0, \frac{1}{2})$.
We set
\begin{equation*}
  \zset \defby C(\R_+; \contisp^{-\alpha})^{\N},
\end{equation*}
which is endowed with the product topology.
We set $\underline{Z} \defby (Z^{:k:})_{k=0}^{\infty} \in \zset$. We also set
\begin{equation*}
  \H \defby \set{h \in C(\R_+; \contisp^1) \given
  h(t) = \int_{-\infty}^t e^{(t-s)(\Delta - 1)} g(s, \cdot) ds \mbox{ for some }
  g \in L^2(\R \times \torus) }.
\end{equation*}
The main theorem of this paper is the following;
\begin{theorem}\label{thm:main}
  We have
  \begin{equation}\label{eq:main}
    \Supp\{\Law(\underline{Z})\} = \overline{\set{(H_k(h, R))_{k=0}^{\infty}
    \given h \in \H, \,\, R \geq 0}}^{\zset}.
  \end{equation}
\end{theorem}
Before moving to the proof of Theorem \ref{thm:main},
we briefly explain a corollary about the dynamic $P(\Phi)_2$ equation \eqref{eq:P_Phi}.
See \cite{TW18} for more detail.

We fix $\alpha_0, \beta \in (0, \frac{1}{N})$.
We set
\begin{equation*}
  \hat{Z}^{:k:}(t) \defby \sum_{l=0}^k \binom{k}{l} Z^{:k-l:}(t) (-e^{t(\Delta - 1)} Z(0) )^l.
\end{equation*}

The equation \eqref{eq:P_Phi} can be solved by the Da Prato-Debussche method.
We say $Y$ is the solution of the shifted equation
\begin{equation*}
  \left\{
  \begin{aligned}
    &\partial_t Y = (\Delta - 1)Y - \sum_{k=0}^N a_k \sum_{l=0}^k \binom{k}{l} \hat{Z}^{:k-l:} Y^l, \\
    &Y(0, \cdot) = x
  \end{aligned}
  \right.
\end{equation*}
with the initial condition $x \in \contisp^{-\alpha_0}$ if we have
$Y \in C((0, \infty); \contisp^{\beta})$ and
\begin{equation*}
   Y(t) = e^{t(\Delta - 1)} x - \int_0^t e^{(t-s)(\Delta - 1)} \left( \sum_{k=0}^N a_k \sum_{l=0}^k \binom{k}{l} \hat{Z}^{:k-l:}(s)
    Y^l(s) \right) ds
\end{equation*}
for every $t \in (0, \infty)$.
As shown in \cite[Section 3]{TW18}, the shifted equation has exactly one solution.
Then we call $\Phi \defby \hat{Z} + Y$ the solution of the dynamic $P(\Phi)_2$ equation \eqref{eq:P_Phi}.

As shown in \cite[Section 5]{TW18}, the solution $\Phi = \{\Phi(t)\}_{t \in \R_+}$ defines a
strong Feller process on $\contisp^{-\alpha_0}$. Let $\{P_t\}_{t \in \R_+}$ be the semigroup generated by
$\Phi$ and $P_t^*$ be the adjoint of $P_t$.
We denote by $\norm{\cdot}_{\operatorname{TV}}$ the total variation norm of signed measures on $\contisp^{-\alpha_0}$.
Now we can state a corollary of Theorem \ref{thm:main}.
\begin{corollary}
  Let $\Phi$ be the solution of the dynamic $P(\Phi)_2$ equation.
  Then there exists a unique invariant measure $\mu$ for the semigroup $\{P_t\}_{t \in \R_+}$,
  and there exists $\lambda \in (0, 1)$ such that
  \begin{equation*}
    \norm{P_t^* \nu - \mu}_{\operatorname{TV}} \leq (1-\lambda)^{t} \norm{\nu - \mu}_{\operatorname{TV}}
  \end{equation*}
  for every $t \geq 3$ and probability measure $\nu$ on $\contisp^{-\alpha_0}$.
\end{corollary}
\begin{proof}
  Once we obtain Theorem \ref{thm:main}, this corollary can be proved by a similar argument
  in \cite[Section 6]{TW18}.
\end{proof}
\subsection{Technical lemmas}\label{subsec:technical}
We recall $H_k(x) = H_k(x, 1)$.
\begin{lemma}\label{lem:sufficient_condition_for_integral_of_H_m}
  Let $N \in \N$.
  Assume that $f \in C^{\infty}(\torus)$ satisfies
  \begin{equation*}
    \int_{\torus} f(x)^{2k} dx = (2k-1)!! \quad \mbox{for } k = 1, \ldots, N.
  \end{equation*}
  Then, we have
  \begin{equation*}
    \int_{\torus} H_{2k}(f(x)) dx = 0 \quad \mbox{for } k=1, \ldots, N.
  \end{equation*}
\end{lemma}
\begin{proof}
  According to Remark \ref{remark:explicit_form_of_hermite}, we have
  \begin{equation*}
    H_{2k}(f(x)) = (2k)! \sum_{l=0}^k \frac{(-1)^l}{2^ll! (2k-2l)!} f(x)^{2(k-l)}
  \end{equation*}
  If the assumption is satisfied, then
  \begin{align*}
    \int_{\torus} H_{2k}(f(x)) dx &= (2k)! \sum_{l=0}^k \frac{(-1)^l}{2^l l!(2k-2l)!}
    \int_{\torus} f(x)^{2(k-l)} dx \\
    &= (2k)! \sum_{l=0}^k \frac{(-1)^l}{2^l l! (2k-2l)!} (2k-2l-1)!! \\
    &= \frac{(2k)!}{2^k} \sum_{l=0}^k \frac{(-1)^l}{l! (k-l)!} = 0. \qedhere
  \end{align*}
\end{proof}

\begin{lemma}\label{lem:existence_of_nice_f}
  For every $N \in \N$, there exists $f \in C^{\infty}(\torus)$ such that
  \begin{equation*}
    \int_{\torus} H_{k}(f(x)) dx = 0 \quad \mbox{for } k=1, 2, \ldots, 2N.
  \end{equation*}
\end{lemma}
\begin{proof}
  We set
  \begin{equation*}
    b_k \defby
    \begin{cases}
      0 &\mbox{for odd }k \\
      (\frac{k}{2}-1)!! &\mbox{for even }k.
    \end{cases}
  \end{equation*}
  According to Lemma \ref{lem:sufficient_condition_for_integral_of_H_m}, it suffices to find a smooth $f$ such that
  \begin{equation*}
    \int_{\torus} f(x)^{k} dx = b_k \quad \mbox{for } k = 1, \ldots, 2N.
  \end{equation*}
  We observe that
  \begin{equation*}
    \int_0^1 (2 \log(x_1^{-1}))^k dx_1 = 2^k \int_0^{\infty} y^k e^{-y} dy = (2k)!!
  \end{equation*}
  and that
  \begin{equation*}
    \int_0^1 \cos^{2k}\left(\pi x_2\right) dx_2 = \frac{(2k-1)!!}{(2k)!!}.
  \end{equation*}
  Therefore, we have
  \begin{equation*}
    \int_{\torus} \left\{ \sqrt{2\log(x_1^{-1})} \cos\left(\pi x_2\right) \right\}^{k} dx_1 dx_2 = b_k.
  \end{equation*}
  However, the function
  \begin{equation*}
    \torus \ni (x_1, x_2) \mapsto \sqrt{2\log(x_1^{-1})} \cos\left(\pi x_2\right)
  \end{equation*}
  is not smooth on the torus $\torus$.

  We set
  \begin{multline*}
    g(a, x)  \defby \sqrt{2\log(x_1^{-1})} \cos\left(\pi x_2\right) \\
    \times \exp\left(-a^2(x_1^{-\frac{1}{2}} + (1-x_1)^{-\frac{1}{2}} + x_2^{-\frac{1}{2}} + (1-x_2)^{-\frac{1}{2}})\right).
  \end{multline*}
  The function $g(a, \cdot)$ is smooth on $\torus$ provided $a \neq 0$.
  Furthermore, the map
  \begin{equation*}
    (-1, 1) \ni a \mapsto \int_{\torus} g(a, x)^{k} dx \in \R
  \end{equation*}
  is continuously differentiable.
  Indeed, we have
  \begin{equation*}
    \abs{\partial_a g(a, x)^{k}} \leq 2k g(0, x)^{k} (x_1^{-\frac{1}{2}} + (1-x_1)^{-\frac{1}{2}}
    +x_2^{-\frac{1}{2}} + (1-x_2)^{-\frac{1}{2}})
  \end{equation*}
  for $a \in (-1, 1)$ and the left hand side is integrable.

  We take $\phi_1, \ldots, \phi_{2N} \in C^{\infty}(\torus)$, which will be determined later, and we set
  \begin{equation*}
    h(a_0, a_1, \ldots, a_{2N}; x) \defby g(a_0, x) + a_1 \phi_1(x) + \cdots + a_{2N} \phi_{2N}(x)
  \end{equation*}
  and for $a_0, a_1, \ldots, a_{2N} \in (-1, 1)$
  \begin{equation*}
    H(a_0, a_1, \ldots, a_{2N}) \defby \left(\int_{\torus} h(a_0, a_1, \ldots, a_{2N}; x)^{k} dx \right)_{k=1}^{2N}.
  \end{equation*}
  Then $H$ is continuously differentiable and $H(0, 0, \ldots, 0) = (b_k)_{k=1}^{2N}$.

  We observe
  \begin{equation}\label{eq:det_of_partial_H}
    \det (\partial_{a_i} H)_{i=1}^{2N} (0, 0, \ldots, 0)
    = \det \left( j \int_0^1 \phi_i(x) g(0, x)^{j-1} dx \right)_{i,j = 1}^{2N}.
  \end{equation}
  If \eqref{eq:det_of_partial_H} is nonzero, the implicit function theorem implies that for sufficiently small
  $a_0 \in (0, 1)$, there exist $a_1, \ldots, a_{2N} \in (-1, 1)$ such that
  $H(a_0, a_1, \ldots, a_{2N}) = (b_k)_{k=1}^{2N}$. Then, we can take $f = h(a_0, a_1, \ldots, a_{2N})$.

  Therefore, it remains to show that for suitably chosen $\phi_1, \ldots, \phi_{2N}$, \eqref{eq:det_of_partial_H} is nonzero.
  We note that
  $g(0,x), \ldots, g(0, x)^{2N-1}$
  are linearly independent in $L^2(\torus)$. Therefore, for each $\epsilon \in (0, 1)$,
  we can find $\phi_1, \ldots, \phi_{2N} \in C^{\infty}(\torus)$ such that
  \begin{equation*}
    \abs*{j \int_0^1 \phi_i(x) g(0, x)^{j-1} dx - \indic_{\{i=j\}}} < \epsilon
  \end{equation*}
  for $i, j \in \{1, \ldots, 2N\}$.
  Then we have
  \begin{equation*}
    \det (\partial_{a_i} H)_{i=1}^{2N} (0, 0, \ldots, 0) = 1 + O(\epsilon)
  \end{equation*}
  and for sufficiently small $\epsilon$, \eqref{eq:det_of_partial_H} is nonzero.
\end{proof}
\begin{lemma}\label{lem:properties_of_h_n}
  Let $N \in \N$. Suppose that we are given a sequence $\{C_n\}_{n=1}^{\infty}$ of nonnegative numbers with
  $C_n = O(\log n)$.
  Then there exists a sequence $\{h_n\}_{n=1}^{\infty} \subset \H$ with the following properties.
  \begin{enumerate}[(i)]
    \item The function $h_n$ is of the form
    \begin{equation*}
      h_n(t, x) = \sqrt{C_n} f_{n, t}(l_n x),
    \end{equation*}
    where $l_n \sim (\log n)^{\log(\log n)}$ and $\sup_{t\in \R_+, n \in \N} \norm{f_{n, t}}_{\contisp^1} < \infty$.
    \item We have
    $\lim_{n \to \infty} \sup_{t \in \R_+} \norm{H_{k}(h_n(t), C_n)}_{\contisp^{-\alpha}} = 0$ for $k = 1, \ldots, N$.
  \end{enumerate}
\end{lemma}
\begin{remark}
  We note that $\{l_n\}$ satisfies
  \begin{equation*}
    \lim_{n \to \infty} \frac{l_n}{(\log n)^{\beta}} = \infty, \quad
    \lim_{n \to \infty} \frac{l_n}{n^{\beta}} = 0
  \end{equation*}
  for every $\beta \in (0, \infty)$.
  This is the only property of $\{l_n\}$ which will be used.
\end{remark}
\begin{proof}
  Let $f$ be the function constructed in Lemma \ref{lem:existence_of_nice_f}.
  We set
  \begin{equation*}
    a_m \defby \int_{\torus} f(x) e^{-2\pi i m \cdot x} dx \quad ( m \in \Z^2 ),
  \end{equation*}
  $l_n \defby \lfloor (\log n)^{\log(\log n)} \rfloor$
  and $\lambda_{n, m} \defby 1 + 4\pi^2 l_n^2 \abs{m}^2$.
  If we set
  \begin{align*}
    h_n(t) &\defby \sqrt{C_n} \int_{-1}^t e^{(t-s)(\Delta - 1)} \left(
    \sum_{m \in \Z^2} a_m \lambda_{n, m} e_{l_n m} \right) ds \qquad \mbox{and}\\
    f_{n, t} &\defby \sum_{m \in \Z^2} a_m(1 - e^{-\lambda_{n, m}(t+1)}) e_m,
  \end{align*}
  the condition (i) is satisfied.

  We next check that the condition (ii) is satisfied.
  As we have
  \begin{equation*}
    \sup_{t \in \R_+}\norm{h_n(t) - \sqrt{C_n} f(l_n \cdot)}_{\contisp^1} \lesssim n^{-1},
  \end{equation*}
  it suffices to show
  \begin{equation}\label{eq:H_k_f_n_tends_to_zero}
    \lim_{n \to \infty} C_n^{\frac{k}{2}} \norm{H_k(f(l_n \cdot))}_{\contisp^{-\alpha}} = 0
  \end{equation}
  for $k = 1, \ldots, N$.
  Since $\int_{\torus} H_k(f(x)) dx = 0$, Lemma \ref{prop:scaling_in_besov}
  implies $\norm{H_k(f(l_n \cdot))}_{\contisp^{-\alpha}}$ $\lesssim l_n^{-\alpha}$.
  By comparing the growth rate of $C_n$ and $l_n$, this leads to \eqref{eq:H_k_f_n_tends_to_zero}.
\end{proof}
For $h \in \H$, we define $T_h: C(\R_+; \contisp^{-\alpha})^{\N} \to C(\R_+; \contisp^{-\alpha})^{\N}$ by
\begin{equation*}
  (T_h \underline{z})_k \defby \sum_{l=0}^k \binom{k}{l} z_l h^{k-l} \quad \mbox{for } \, \underline{z} = (z_k)_{k=0}^{\infty}.
\end{equation*}
As $\H \subset C(\R_+ ; \contisp^1)$, Proposition \ref{prop:paraproduct_estimates} implies that $T_h$ is homeomorphic.
\begin{lemma}\label{lem:cameron_martin}
  For every $h \in \H$, we have
  \begin{equation*}
    \Supp\{\Law( \underline{Z} )\}
    = \Supp\{\Law( T_h \underline{Z})\}.
  \end{equation*}
\end{lemma}
\begin{proof}
  Take $g \in L^2(\R \times \torus)$ satisfying
  \begin{equation*}
    h(t) = \int_{-\infty}^t e^{(t-s)(\Delta-1)} g(s, \cdot) ds.
  \end{equation*}
  We define a probability measure $\P_g$ by
  \begin{equation*}
    \frac{d\P_g}{d\P} \defby \exp\left( - \xi(g) - \frac{\norm{g}_{L^2(\R \times \torus)}}{2} \right).
  \end{equation*}
  Since the Cameron-Martin space of the space-time white noise $\xi$ is $L^2(\R \times \torus)$,
  we have $\Law_{\P_g}(\xi + g) = \Law_{\P} (\xi)$.
  Let $\rho \in \S(\R \times \R^2)$ and set $\rho_n(t, x) \defby n^4 \rho(n^2 t, n x)$,
  \begin{equation*}
    \tilde{Z}_n(t) \defby \int_{-\infty}^t e^{(t-s)(\Delta - 1)}[ \xi*\rho_n(s, \cdot)] ds \quad \mbox{and}\quad
    \tilde{Z}_n^{:k:}(t) \defby H_k(\tilde{Z}_n(t), \tilde{\rconst}_n),
  \end{equation*}
  where
  \begin{equation*}
    \tilde{\rconst}_n \defby \int_{\R \times \torus} \abs{K*\rho_n(t, x)}^2 dt dx.
  \end{equation*}
  Similarly, we set
  \begin{equation*}
     h_n(t) \defby \int_{-\infty}^t e^{(t-s)(\Delta-1)}\big[{g*\rho_n}(s, \cdot)\big] ds, \quad
     T_h \tilde{Z}_n(t) \defby \tilde{Z}_n(t) + h_n(t)
  \end{equation*}
  and $T_h \tilde{Z}_n^{:k:}(t) \defby H_k(T_h \tilde{Z}_n(t), \rconst_n)$.
  Then we have
  \begin{equation*}
    \Law_{\P}((\tilde{Z}_n^{:k:})_{k=0}^{\infty}) = \Law_{\P_g}((T_h \tilde{Z}_n^{:k:})_{k=0}^{\infty}).
  \end{equation*}
  As mentioned in Remark \ref{rem:space_time_approximation_of_she}, we have
  $\tilde{Z}_n^{:k:} \to Z^{:k:}$ in $L^2(\P; C(\R_+; \contisp^{-\alpha}))$, and thus we have
  \begin{align*}
    T_h \tilde{Z}_n^{:k:} &= \sum_{l=0}^k \binom{k}{l} H_l(\tilde{Z}_n, \rconst_n) h_n^{k-l} \\
    &\to \sum_{l=0}^k \binom{k}{l} Z^{:l:} h \quad \mbox{in } L^2(\P) \\
    &= T_h Z^{:k:}.
  \end{align*}
  Since $\P$ and $\P_g$ are equivalent, we see that
  $\Law_{\P}(\underline{Z}) = \Law_{\P_g}(T_h \underline{Z})$.

  Now assume $\P( \underline{Z} \in A) > 0$. Then we have
  $\P_g(T_h \underline{Z} \in A) > 0$. Since $\P$ and $\P_g$ are equivalent, we obtain
  $\P(T_h \underline{Z} \in A) > 0$. As the converse similarly holds, we complete the proof.
\end{proof}

\subsection{Proof of the main theorem}\label{subsec:proof_of_main_thm}
\begin{lemma}\label{lem:key_lemma}
  For every $R \in \R_+$, we have
  \begin{equation*}
    (H_k(0, R))_{k=0}^{\infty} \in \Supp\{\Law(\underline{Z})\}.
  \end{equation*}
\end{lemma}
\begin{proof}
  Set
  $Z_n(t) \defby \Pi_n Z(t)$, $\rconst_n \defby \expect[Z_n(t, x)^2]$
  and $C_n \defby \max\{\rconst_n - R, 0\}$.
  We fix $N \in \N$.
  Let $\{h_n\}_{n=1}^{\infty}$ be the sequence of functions constructed in Lemma \ref{lem:properties_of_h_n}
  corresponding to $N$ and $\{C_n\}_{n=1}^{\infty}$.

  \emph{STEP 1.} We first prove
  \begin{equation*}
    \lim_{n \to \infty} T_{-Z_n - h_n}Z^{:k:} = H_k(0, R) \quad \mbox{for } k=1, \ldots, N.
  \end{equation*}
  in probability.
  Strictly speaking, this is an abuse of notation as $Z_n \notin \H$.
  We come back to this problem in \emph{STEP 2}.

  We compute
  \begin{align}
    T_{-Z_n-h_n} Z^{:k:} ={}& \sum_{l=0}^k \binom{k}{l} (-1)^l Z^{:k-l:} (h_n + Z_n)^l \nonumber\\
    ={}& \sum_{l=0}^k \binom{k}{l} (-1)^l Z^{:k-l:} \sum_{j=0}^l \binom{l}{j} h_n^j Z_n^{l-j} \nonumber\\
    ={}& \sum_{j=0}^k h_n^j \sum_{l=j}^k (-1)^l \binom{k}{l} \binom{l}{j} Z^{:k-l:} Z_n^{:l-j:} \nonumber \\
    \begin{split}\label{eq:sum_of_h_and_etc}
      ={}&  \sum_{j=0}^k h_n^j \sum_{l=j}^k (-1)^l \binom{k}{l} \binom{l}{j} Z_n^{:k-l:} Z_n^{:l-j:} \\
      & + \sum_{j=0}^k h_n^j \sum_{l=j}^k (-1)^l \binom{k}{l} \binom{l}{j} (Z^{:k-l:} - Z_n^{:k-l:}) Z_n^{l-j}.
    \end{split}
  \end{align}
  We evaluate the first term of \eqref{eq:sum_of_h_and_etc}.
  Applying Proposition \ref{prop:hermite_binomial_expansion}, we observe
  \begin{align*}
    \sum_{l=j}^k (-1)^l \binom{k}{l} \binom{l}{j} Z_n^{:k-l:} Z_n^{l-j}
    &= \sum_{l=0}^{k-j} (-1)^{j+l} \frac{k!}{(k-j-l)! j! l!} H_{k-j-l}(Z_n, \rconst_n) Z_n^l \\
    &=\frac{(-1)^j k!}{j!(k-j)!} H_{k-j}(0, \rconst_n),
  \end{align*}
  and hence, applying Proposition \ref{prop:hermite_binomial_expansion} again,
  the first term of \eqref{eq:sum_of_h_and_etc} equals
  \begin{align*}
    \sum_{j=0}^k (-1)^j \binom{k}{j} h_n^j H_{k-j}(0, \rconst_n)
    &= H_k(-h_n, \rconst_n) \\
    &= \frac{1}{k!} \left. \left( \frac{d}{dt} \right)^k \right\rvert_{t=0}
    \exp\left(-\frac{t^2 R}{2} \right) \exp\left(-t h_n - \frac{C_n t^2}{2} \right),
  \end{align*}
  which converges to $H_k(0, R)$ by Lemma \ref{lem:properties_of_h_n}.

  We next show the second term of \eqref{eq:sum_of_h_and_etc} converges to $0$, for which it suffices to show
  \begin{equation*}
    \lim_{n \to \infty} \sup_{0 \leq t \leq T}
    \norm{h_n^j(t) Z_n^k(t) (Z^{:l:}(t) - Z^{:l:}_n(t))}_{\contisp^{-\alpha}} = 0
  \end{equation*}
  in probability.
  Let $h_n(t, x) = \sqrt{C_n} f_{n, t}(l_n x)$ be the representation given in
  the condition (i) of Lemma \ref{lem:properties_of_h_n}.
  Then we have
  \begin{align*}
    \norm{h_n^j(t) Z_n^k(t)(Z^{:l:}(t) - Z_n^{:l:}(t))}_{\contisp^{-\alpha}}
      &\lesssim C_n^{\frac{j}{2}} \norm{f_{n,t}(l_n \cdot)}_{\contisp^1}^j
      \norm{Z_n^k(t) (Z^{:l:}(t) - Z_n^{:l:}(t))}_{\contisp^{-\alpha}} \\
      &\lesssim C_n^{\frac{j}{2}} l_n^{j}
      \norm{Z_n^k(t) (Z^{:l:}(t) - Z_n^{:l:}(t))}_{\contisp^{-\alpha}}.
  \end{align*}
  The last inequality is derived by using Proposition \ref{prop:scaling_in_besov}-(i).
  Therefore, it comes down to proving
  \begin{equation*}
    \lim_{n \to \infty} l_n^j \expect[ \sup_{0 \leq t \leq T} \norm{Z_n^k(t)
    (Z^{:l:}(t) - Z_n^{:l:}(t))}_{\contisp^{-\alpha}}^2 ] = 0
  \end{equation*}
  for every $j, k, l \in \N$.
  However, this easily follows from Proposition \ref{prop:hermite_inversion} and Lemma \ref{lem:fin_dim_approxim_of_she}.

  \emph{STEP 2.} The problem is that $Z_n$ is not an element of $\H$.
  Let $W_{m, n}$ be a smooth approximation of $W_m$ which satisfies
  \begin{equation}\label{eq:smooth_approximation_of_W_m}
    \expect[ \sup_{t \in [-n, T]} \abs{W_m(t) - W_{m, n}(t)}^2 ] \leq 2^{-n}.
  \end{equation}
  See Remark \ref{rem:bm_from_white_noise} for the definition of $W_m$.
  We set
  \begin{equation*}
    \tilde{Z}_n(t) \defby \sum_{\abs{m} \leq n} \left(\int_{-n}^t e^{-(1+4\pi^2 \abs{m}^2)(t-s)} dW_{m,n}(s)
    \right)  e_m.
  \end{equation*}
  Note that $\tilde{Z}_n \in \H$.
  We claim
  \begin{equation}\label{eq:Z_and_Z_tilde}
    \lim_{n \to \infty} \expect[\sup_{0 \leq t \leq T} \norm{Z_n(t) - \tilde{Z}_n(t)}_{\contisp^1}^2] = 0.
  \end{equation}
  We have
  \begin{multline*}
    Z_n(t) - \tilde{Z}_n(t) \\
    = \sum_{\abs{m} \leq n} \int_{-\infty}^{-n} e^{(t-s)(\Delta-1)} e_m dW_m(s)
    + \sum_{\abs{m}\leq n} \int_{-n}^t e^{(t-s)(\Delta-1)} e_m d(W_m(s) - W_{m, n}(s)).
  \end{multline*}
  The $L^2(\P)$-norm of the first term is bounded by
  \begin{align*}
    \MoveEqLeft
    \expect\Big[ \sup_{0 \leq t \leq T} \Big\lvert \sum_{\abs{m}\leq n} \int_{-\infty}^{-n}
    e^{-(1+4\pi^2 \abs{m}^2)(t-s)} dW_m(s) \norm{e_m}_{\contisp^1}\Big\rvert^2 \Big] \\
    &\lesssim \sum_{\abs{m} \leq n} (1+\abs{m})^2 \expect
    \Big[ \sup_{0\leq t \leq T} \Big\lvert \int_{-\infty}^{-n} e^{-(1+4\pi^2 \abs{m}^2)(t-s)} dW_m(s) \Big\rvert^2 \Big] \\
    &\lesssim \sum_{\abs{m} \leq n} (1+\abs{m})^2 \int_{-\infty}^{-n} e^{2(1+4\pi^2\abs{m}^2)s} ds \\
    &\lesssim \sum_{\abs{m} \leq n} e^{-2(1+4\pi^2 \abs{m}^2)n},
  \end{align*}
  which converges to $0$.

  The second term equals
  \begin{multline*}
    \sum_{\abs{m} \leq n} \Big\{
    W_m(t) - W_{m, n}(t)
    -e^{-(1+4\pi^2 \abs{m}^2)(t+n)}(W_m(-n) - W_{m, n}(-n)) \\
    +(1+4\pi^2 \abs{m}^2) \int_{-n}^t e^{-(1+4\pi^2 \abs{m}^2)(t-s)}
    (W_m(s) - W_{m,n}(s)) ds \Big\} e_m
  \end{multline*}
  Using \eqref{eq:smooth_approximation_of_W_m}, calculation similar to the first term yields
  \begin{equation*}
    \lim_{n \to \infty} \expect \Big[ \sup_{0 \leq t \leq T}
     \Big\lVert \sum_{\abs{m}\leq n} \int_{-n}^t e^{(t-s)(\Delta-1)} e_m d(W_m(s) - W_{m, n}(s))
     \Big\rVert_{\contisp^{1}}^2 \Big] = 0.
  \end{equation*}
  Now the proof of \eqref{eq:Z_and_Z_tilde} is complete.

  \emph{STEP 3.} We have
  \begin{equation*}
    T_{-\tilde{Z}_n - h_n}Z^{:k:}
    = T_{-Z_n-h_n} Z^{:k:} + \sum_{l=1}^k \binom{k}{l} (Z_n - \tilde{Z}_n)^l T_{-Z_n-h_n}Z^{:k-l:}.
  \end{equation*}
  Therefore, $\lim_{n \to \infty} T_{-\tilde{Z}_n - h_n} Z^{:k:} = H_k(0, R)$ in probability.

  Take $\underline{z} = (z_k)_{k=0}^{\infty} \in \Supp\{\Law(\underline{Z})\}$.
  Then there exists a sequence $\{\tilde{h}_n\}_{n=1}^{\infty}\subset \H$ such that
  \begin{equation*}
    \lim_{n \to \infty} T_{\tilde{h}_n} \underline{z} = (H_k(0, R))_{k=0}^{\infty} \quad \mbox{in } \zset.
  \end{equation*}
  Since Lemma \ref{lem:cameron_martin} implies $T_{\tilde{h}_n} \underline{z} \in \Supp\{\Law(\underline{Z})\}$ and
  the support of a measure is closed, we conclude
  $(H_k(0, R))_{k=0}^{\infty} \in \Supp\{\Law(\underline{Z})\}$.
\end{proof}

\begin{proof}[Proof of Theorem \ref{thm:main}]
  We denote by $\mathcal{X}$ the right hand side of \eqref{eq:main}.
  Let
  $Z_n \defby \Pi_n Z$, $\rconst_n \defby \expect[ Z_n(t, x)^2 ]$ and
  $Z_n^{:k:} \defby H_k(Z_n, \rconst_n)$. Then we have $\lim_{n \to \infty} Z_n^{:k:} = Z^{:k:}$.
  Although $Z_n \notin \mathcal{X}$, a trick similar to \emph{STEP 2} of Lemma \ref{lem:key_lemma} allows us to show
  \begin{equation*}
    \Supp\{ \Law(\underline{Z}) \} \subset \mathcal{X}.
  \end{equation*}

  We move to prove the other inclusion. As
  $(H_k(0, R))_{k=0}^{\infty} \in \Supp\{\Law(\underline{Z})\}$ by Lemma \ref{lem:key_lemma},
  for every $h \in \H$ and $R \geq 0$, we have
  \begin{equation*}
    (H_k(h, R))_{k=1}^{\infty} \in \Supp\{ \Law(T_h \underline{Z})\}.
  \end{equation*}
  Lemma \ref{lem:cameron_martin} implies $(H_k(h, R))_{k=0}^{\infty} \in \Supp\{\Law(\underline{Z})\}$.
  Since the support of a measure is closed, we complete the proof.
\end{proof}

\subsection{Complex version}\label{subsec:complex}
In this subsection, all functions are allowed to be complex-valued.
We can consider a complex version of the additive stochastic heat equation
\begin{equation*}
  \partial_t Z_{\C} = \mu (\Delta - 1) Z_{\C} + \xi_{\C}.
\end{equation*}
Here $\Re(\mu) \in (0, \infty)$ and $\xi_{\C}$ is a complex space-time white noise.
For more information, see \cite{Mat19} or \cite{Tre19}.
Rigorously, $Z_{\C}$ is a $C(\R_+; \contisp^{-\alpha})$-valued random variable such that
for $\phi \in C^{\infty}(\torus; \C)$ we have
\begin{equation*}
  \inp{Z_{\C}(t)}{\phi} = \int_{\R \times \torus} \inp{\kernel_{\C, \mu}(t-s, \cdot - y)}{\phi} \xi_{\C}(ds dy)
  \quad \mbox{almost surely,}
\end{equation*}
where the integral is in the sense of complex It\^o-Wiener integral and
\begin{equation*}
  \kernel_{\C, \mu}(t, x) \defby \sum_{m \in \Z^2} e^{-\mu(1+ 4 \pi^2 \abs{m}^2)t} e^{2 \pi i m\cdot x}.
\end{equation*}
To define Wick powers of $Z_{\C}$, we introduce complex Hermite polynomials
\begin{equation*}
  H_{k,l}(z, c) \defby \sum_{m=0}^{\min\{k, l\}} m! \binom{k}{m} \binom{l}{m} (-c)^m z^{k-m} \conj{z}^{l-m}.
\end{equation*}
If we set $\rconst_n \defby \expect[ \abs{\Pi_n Z_{\C}(t) (x)}^2]$, then
\begin{equation*}
  H_{k,l}(\Pi_n Z_{\C}, \rconst_n) \mbox{ converges in } L^p(\P; C([0, T]; \contisp^{-\alpha})).
\end{equation*}
We denote the limit by $Z_{\C}^{:k, l:}$.
As in the real case, we can determine the support of the law of $(Z_{\C}^{:k,l:})_{k,l \in \N}$.
\begin{theorem}\label{thm:support_complex}
  Set
  \begin{equation*}
    \H_{\C} \defby \set{h \given h(t) = \int_{-\infty}^t e^{\mu(t-s) (\Delta - 1)} g(s, \cdot) ds
    \mbox{ for some } g \in L^2(\R \times \torus; \C)}.
  \end{equation*}
  Then we have
  \begin{equation*}
    \Supp\{ \Law((Z_{\C}^{:k,l:})_{k,l \in \N}) \}
    = \overline{\set{(H_{k,l}(h, R))_{k,l \in \N} \given h \in \H_{\C}, R \in \R_+}}.
  \end{equation*}
\end{theorem}
We will not provide a proof as it is parallel to the real case. We only explain how to modify
Lemma \ref{lem:existence_of_nice_f}.
For the complex case, we need to find a smooth $f$ with
\begin{equation}\label{eq:nice_f_complex}
  \int_{\torus} H_{k, l}(f(x), 1) dx = 0 \quad \mbox{for } (k, l) \in \{0, 1, \ldots, N\}^2 \setminus \{(0, 0)\}.
\end{equation}
We observe $f(x) \defby \sqrt{\log(x_1^{-1})} e^{2 \pi i x_2}$ satisfies \eqref{eq:nice_f_complex} and hence
we modify $f$ to become smooth as in Lemma \ref{lem:existence_of_nice_f}.

Finally, we note that Theorem \ref{thm:support_complex} can be used to prove the support theorem,
hence the exponential ergodicity of the complex Ginzburg-Landau equation studied in
\cite{Mat19} and \cite{Tre19}.

\section{Support theorem: Gaussian multiplicative chaos}\label{sec:gmc}
In this section, we use Sobolev spaces $\sobolev^{\gamma}$. They are Hilbert spaces with inner products defined by
\begin{equation*}
  \inp{\phi}{\psi}_{\sobolev^{\gamma}}
  = \sum_{m \in \Z^2} (1 + 4\pi^2 \abs{m}^2)^{\gamma} \inp{\phi}{e_m} \conj{\inp{\psi}{e_m}}.
\end{equation*}
It is well-known that $\sobolev^{\gamma}$ is isomorphic to $\besov^{\gamma}_{2, 2}$.
See \cite[Page 99]{Bahouri2011}.

Let $\{Z(t)\}_{t \in \R_+}$ be the stationary solution of the additive stochastic heat equation as before.
The random variable $Z(t)$ is identified with a massive Gaussian free field $X$, i.e. a $\contisp^{-\alpha}$-valued
centered Gaussian vector with covariance structure
\begin{equation*}
  \expect[ \inp{X}{\phi} \inp{X}{\psi} ] = \int_{\torus \times \torus} \green(x, y) \phi(x) \psi(y) dx dy
  \quad \phi, \psi \in C^{\infty}(\torus),
\end{equation*}
where
\begin{equation*}
  \green(x, y) \defby \frac{1}{2} \sum_{m \in \Z^2} \frac{e_m(x) \conj{e_m}(y)}{1 + 4 \pi^2 \abs{m}^2}
\end{equation*}
is the Green's function with respect to $1 - \Delta$.
In particular, the Cameron-Martin space of $X$ is $\sobolev^1$.

As in the case of the additive stochastic heat equation, we can define Wick powers $X^{:k:}$ of $X$.
Theorem \ref{thm:main} implies the following;
\begin{corollary}\label{cor:support_of_gff}
  Let $X$ be the massive Gaussian free field as above. Then we have
  \begin{equation*}
    \Supp\{\Law((X^{:k:})_{k=0}^{\infty})\}
    = \overline{\set{(H_k(h, R))_{k=0}^{\infty} \given h \in C^{\infty}(\torus), R \geq 0}}^{(\contisp^{-\alpha})^{\N}}.
  \end{equation*}
\end{corollary}

Fix $\abs{\gamma} < \sqrt{8 \pi}$ and $\beta \in (\frac{\gamma^2}{8 \pi}, 1)$.
Set $X_n \defby \Pi_n X$ and $\rconst_n \defby \expect[X_n(x)^2]$.
Then, as shown in \cite[Theorem 2.2]{HKK19},
\begin{equation*}
  \wick{e^{\gamma X_n}} \,\, \defby e^{\gamma X_n - \frac{\gamma^2}{2} \rconst_n}
  = \sum_{k=0}^{\infty} \frac{\gamma^k}{k!} X_n^{:k:}
\end{equation*}
converges in $L^2(\P; \sobolev^{\beta})$. The limit, denoted by $\wick{e^{\gamma X}}$,
is called Gaussian multiplicative chaos.
A strategy similar to the proof of Theorem \ref{thm:main} enables us to determine the support of the law of
$\wick{e^{\gamma X}}$.

For a smooth function $h$ on $\torus$, we define $T_h: \sobolev^{-\beta} \to \sobolev^{-\beta}$ by
$T_h f \defby e^{\gamma h} f$.
By Proposition \ref{prop:paraproduct_estimates}, $T_h$ is homeomorphic.
\begin{lemma}\label{lem:cameron_martin_for_gmc}
  For every smooth function $h$ on $\torus$, we have
  \begin{equation*}
    \Supp\{\Law(T_h \wick{e^{\gamma X}} )\} = \Supp\{\Law(\wick{e^{\gamma X}})\}.
  \end{equation*}
\end{lemma}
\begin{proof}
  We define a probability measure $\P_h$ by
  \begin{equation*}
    \frac{d \P_h}{d \P} \defby \exp\Big( -\inp{X}{h}_{\sobolev^1} - \frac{\norm{h}_{\sobolev^1}^2}{2} \Big).
  \end{equation*}
  Then we have $\Law_{\P_h}( X + h ) = \Law_{\P} (X)$. The remainder of the proof is the same as the proof of
  \ref{lem:cameron_martin}.
\end{proof}
\begin{lemma}\label{lem:1_is_in_the_support_of_gmc}
  We have $1 \in \Supp\{\Law(\wick{e^{\gamma X}})\}$.
\end{lemma}
\begin{proof}
  We define $X_N, \rconst_N$ as above and set $h_N \defby - \frac{\gamma \rconst_N}{2}$.
  We observe
  \begin{equation*}
    T_{-h_N - X_N} \wick{e^{\gamma X}} = \lim_{M \to \infty} \wick{e^{\gamma (X_M - X_N)}}
  \end{equation*}
  in $L^2(\P; \sobolev^{-\beta})$.
  The key is to prove
  \begin{equation}\label{eq:e^gamma_to_1}
    \lim_{N \to \infty} \lim_{M \to \infty} \expect[ \norm{\wick{e^{\gamma(X_M - X_N)}}-1}_{\sobolev^{-\beta}}^2 ] = 0.
  \end{equation}

  The following computation is in the spirit of \cite[Theorem 2.2]{HKK19}.
  Set $\mu_m \defby 1 + 4\pi^2 \abs{m}^2$ for $m \in \Z^2$. We have
  \begin{equation*}
    \expect[ \norm{\wick{e^{\gamma(X_M - X_N)}} - 1}_{\sobolev^{-\beta}}^2 ]
    = \sum_{m \in \Z^2} \mu_m^{-\beta} \expect[ \abs{\inp{\wick{e^{\gamma(X_M -X_N)}} - 1}{e_m}}^2]
  \end{equation*}
  and
  \begin{equation*}
    \expect[\abs{\inp{\wick{e^{\gamma(X_M-X_N)}}-1}{e_m}}^2]
    = \sum_{k=1}^{\infty} \frac{\gamma^{2k}}{(k!)^2} \expect[ \abs{\inp{(X_M - X_N)^{:k:}}{e_m}}^2].
  \end{equation*}
  We observe
  \begin{align*}
    \MoveEqLeft \expect[ \abs{\inp{(X_M - X_N)^{:k:}}{e_m}}^2] \\
    &= \int_{\torus \times \torus} dx dy e_m(x) \conj{e_m}(y) \expect[ (X_M(x) - X_N(x))^{:k:} (X_M(y) - X_N(y))^{:k:}] \\
    &= k!\int_{\torus \times \torus} dx dy e_m(x) \conj{e_m}(y) \big(\expect[ (X_M(x) - X_N(x)) (X_M(y) - X_N(y))] \big)^k \\
    &= k! \int_{\torus \times \torus} dx dy e_m(x) \conj{e_m}(y) \Big( \sum_{N < \abs{n} \leq M} \frac{e_n(x) \conj{e_n}(y)}
    {2\mu_n} \Big)^k \\
    &= k! \sum_{\substack{N < \abs{n_1}, \ldots, \abs{n_k} \leq M \\ n_1 + \cdots + n_k = m}}
    \prod_{j=1}^k \frac{1}{2\mu_{n_j}}.
  \end{align*}

  Therefore, we obtain
  \begin{equation*}
    \lim_{M \to \infty} \expect[ \norm{\wick{e^{\gamma(X_M - X_N)}}-1}_{\sobolev^{-\beta}}^2 ]
    = \sum_{m \in \Z^2} \sum_{k=1}^{\infty} \frac{\gamma^{2k}}{k!} \mu_m^{-\beta}
    \sum_{\substack{N < \abs{n_1}, \ldots, \abs{n_k} \\ n_1 + \cdots + n_k = m}} \prod_{j=1}^k (2\mu_{n_j})^{-1},
  \end{equation*}
  which equals
  \begin{equation}\label{eq:integral_of_green_functions}
    \sum_{k=1}^{\infty} \frac{\gamma^{2k}}{k!} \int_{\torus \times \torus}
    \Big\{ \sum_{m \in \Z^2} \mu_m^{-\beta} e_m(x) \conj{e_m}(y) \Big\}
    \Big\{ \frac{1}{2} \sum_{N < \abs{m}} \mu_m^{-1} e_m(x) \conj{e_m}(y) \Big\}^k dxdy.
  \end{equation}
  If we set
  \begin{equation*}
    \green_{\beta}(x, y) \defby  \sum_{m \in \Z^2} \mu_m^{-\beta} e_m(x) \conj{e_m}(y)
    \quad \mbox{and} \quad \green^{(N)}(x, y) \defby \frac{1}{2} \sum_{\abs{m} \leq N} \mu_m^{-1} e_m(x) \conj{e_m}(y),
  \end{equation*}
  \eqref{eq:integral_of_green_functions} equals
  \begin{equation}\label{eq:limit_of_greens}
     \int_{\torus \times \torus} \green_{\beta}(x, y) (e^{\gamma^2(\green(x, y) - \green^{(N)}(x, y))} - 1) dxdy.
  \end{equation}

  Hence the proof of \eqref{eq:e^gamma_to_1} comes down to proving convergence of \eqref{eq:limit_of_greens} to $0$.
  According to \cite[Lemma 5.2]{MR99}, we have
  \begin{align*}
    0 \leq \green_{\beta}(x, y) \lesssim_{\beta} 1 + \abs{x-y}^{2\beta - 2} \quad \mbox{and} \quad
    0 \leq \green(x, y) \leq C + \frac{1}{4\pi}\log_+( \abs{x-y}^{-1} ).
  \end{align*}
  If we set $\nu_{\beta}(dxdy) \defby \green_{\beta}(x, y) dx dy$, $\nu_{\beta}$ is a finite measure on $\torus \times \torus$.
  Furthermore, since $\lim_{N \to \infty} \green^{(N)} = \green$ in $L^2(\torus \times \torus, dxdy)$, we have
  \begin{equation*}
    \lim_{N \to \infty} e^{\gamma^2 (\green - \green^{(N)})} - 1 = 0 \quad \mbox{in $\nu_{\beta}$-measure.}
  \end{equation*}
  Thanks to Vitali's convergence theorem, the proof of \eqref{eq:limit_of_greens}
  converging to $0$ is complete if we show
  \begin{equation}\label{eq:finite_integral_green}
    \sup_{N \geq 1} \int_{\torus \times \torus} e^{\lambda \gamma^2 (\green(x,y) - \green^{(N)}(x, y))}  \nu_{\beta}(dxdy) < \infty
  \end{equation}
  for some $\lambda \in (1, \infty)$.

  We compute
  \begin{align*}
    \MoveEqLeft
     \int_{\torus \times \torus} e^{\lambda \gamma^2 (\green(x, y) - \green^{(N)}(x, y))} \nu_{\beta}(dxdy) \\
    &= \sum_{m \in \Z^2} \sum_{k=0}^{\infty} \frac{(\lambda \gamma^2)^k}{k!} \mu_m^{-\beta}
    \sum_{\substack{N < \abs{n_1}, \ldots, \abs{n_k} \\ n_1 + \cdots + n_k = m}} \prod_{j=1}^k (2\mu_{n_j})^{-1} \\
    &\leq \sum_{m \in \Z^2} \sum_{k=0}^{\infty} \frac{(\lambda \gamma^2)^k}{k!} \mu_m^{-\beta}
    \sum_{ n_1 + \cdots + n_k = m} \prod_{j=1}^k (2\mu_{n_j})^{-1} \\
    &=  \int_{\torus \times \torus} e^{\lambda \gamma^2 \green(x,y)} \green_{\beta}(x,y) dxdy \\
    &\lesssim \int_{\torus \times \torus} e^{\lambda \gamma^2\{ C + \frac{1}{4\pi} \log_+(\abs{x-y}^{-1})\}}
    ( 1 + \abs{x-y}^{2\beta - 2}) dxdy \\
    &\lesssim \int_{\torus \times \torus} \abs{x-y}^{-\frac{\lambda \gamma^2}{4 \pi} + 2\beta - 2} dx dy.
  \end{align*}
  This is finite provided $\frac{\lambda \gamma^2}{4 \pi} + 2 - 2 \beta < 2$, or
  $\lambda \in (1, 8\pi \beta \gamma^{-2})$.
  Thus we complete the proof of \eqref{eq:finite_integral_green}, hence the proof of \eqref{eq:e^gamma_to_1}.

  Now suppose $\mu \in \Supp\{\Law(\wick{e^{\gamma X}})\}$. The above argument shows that there exists a sequence
  $\{\tilde{h}_n\}_{n=1}^{\infty} \subset C^{\infty}(\torus)$ with
  $T_{\tilde{h}_n} \mu$ converging to $1$ in $\sobolev^{-\beta}$.
  By Lemma \ref{lem:cameron_martin_for_gmc}, we conclude $1 \in \Supp\{\Law(\wick{e^{\gamma X}})\}$.
\end{proof}
\begin{theorem}\label{thm:support_of_gmc}
  We have
  \begin{equation}\label{eq:support_of_gmc}
    \Supp\{\Law(\wick{e^{\gamma X}})\} = \set{f \in \sobolev^{-\beta} \given f \geq 0}.
  \end{equation}
\end{theorem}
\begin{proof}
    Let $\mathcal{Y}_1$ and $\mathcal{Y}_2$ be the left hand side and the right hand side of \eqref{eq:support_of_gmc}
    respectively. We set $X_N \defby \Pi_N X$ and $\rconst_N \defby \expect[ X_N(x)^2 ]$ as before.
    Since
    \begin{equation*}
      \lim_{N \to \infty} e^{\gamma X_N - \frac{\gamma^2 \rconst_N}{2}} = \wick{e^{\gamma X}}
      \quad \mbox{in } L^2(\P; \sobolev^{-\beta}),
    \end{equation*}
    we have $\mathcal{Y}_1 \subset \mathcal{Y}_2$.

    Let $h \in C^{\infty}(\torus)$ and $U$ be a neighborhood of $e^{\gamma h}$.
    Since $T_{-h} U$ is a neighborhood of $1$, we have
    \begin{equation*}
      \P( \wick{e^{\gamma X}} \in U) = \P( T_{-h} \wick{e^{\gamma X}} \in T_{-h} U) > 0
    \end{equation*}
    by Lemma \ref{lem:cameron_martin_for_gmc} and Lemma \ref{lem:1_is_in_the_support_of_gmc}. Therefore $e^{\gamma h} \in \mathcal{Y}_1$, hence
    $\mathcal{Y}_2 \subset \mathcal{Y}_1$.
\end{proof}
\begin{remark}
  We only consider the case $\gamma \in (-\sqrt{8\pi}, \sqrt{8\pi})$, where
  the Gaussian multiplicative chaos has the second moment.
  However, the theory of Gaussian multiplicative chaos extends to
  the case $\gamma \in (-\sqrt{16 \pi}, \sqrt{16\pi})$. See \cite{Kah85} or \cite{RV14}.
  We expect similar results hold in this general setting.
  The first step is to identify a natural Banach space $\wick{e^{\gamma X}}$ lives in.
  We leave this for a future study.
\end{remark}
\appendix
\section{Estimates of the additive stochastic heat equation}\label{sec:estimate_of_she}
In this appendix, we provide technical estimates of the additive stochastic heat equation
which are needed in Section \ref{sec:support}.
\begin{definition}
  For symmetric functions $K_1, K_2: \Z^2 \to \R_+$, we set
  \begin{equation*}
    K_1 \conv K_2 (m) \defby \sum_{l \in \Z^2} K_1(m-l) K_2(l)
  \end{equation*}
  and for $N \in \N \cup \{\infty\}$ we set $(K_1)_n(m) \defby K_1(m) \indic_{\{\abs{m} \leq N\}}$.
\end{definition}
\begin{lemma}\label{lem:kernel_estimate}
  Suppose $\alpha, \beta \in (0, 1)$ satisfy $\alpha + \beta > 1$.
  Let $K_1, K_2:\Z^2 \to \R_+$ be symmetric functions such that
  \begin{equation*}
    K_1(m) \leq \frac{C}{(1+\abs{m}^2)^{\alpha}},
    \qquad K_2(m) \leq \frac{C}{(1+\abs{m}^2)^{\beta}}.
  \end{equation*}
  Then, there exists a constant $C' = C'(\alpha, \beta, C) \in (0, \infty)$ such that
  \begin{equation*}
    K_1 \conv K_2 (m) \leq \frac{C'}{(1+\abs{m}^2)^{\alpha+\beta-1}}
  \end{equation*}
  and
  \begin{equation*}
    \abs{K_1 \star K_2 (m) - K_1 \star (K_2)_n(m)} \leq \frac{C'}{(1+\max\{\abs{m}, N\}^2)^{\alpha+\beta-1}}
  \end{equation*}
  for every $m \in \Z^2$ and $N \in \N$.
\end{lemma}
\begin{proof}
  See \cite[Lemma C.2]{TW18}.
\end{proof}
\begin{lemma}\label{lem:fin_dim_approxim_of_she}
  Let $T \in (0, \infty)$. For every $k, l \in \N$, there exists $\epsilon \in (0, 1)$ such that
  \begin{equation*}
    \expect[ \sup_{0 \leq t \leq T} \norm{Z_n^{:k:}(t) (Z_m^{:l:}(t) - Z_n^{:l:}(t))}_{\contisp^{-\alpha}}^2 ]
    \lesssim_{\alpha, T, k, l} n^{-\epsilon}
  \end{equation*}
  for every $m > n$.
\end{lemma}
\begin{proof}
  Although the proof is in the spirit of \cite[Appendix E]{TW18},
  we provide a complete proof below.
  We first note that it suffices to prove
  \begin{equation*}
    \expect[ \sup_{0 \leq t \leq T} \norm{Z^{:k:}_n(t) \wickproduct (Z_m^{:l:}(t) - Z_n^{:l:}(t))}_{\contisp^{-\alpha}} ]
    \lesssim n^{-\epsilon}
  \end{equation*}
  for some $\epsilon \in (0, 1)$, since the multiplication formula(\cite[\largethm 7.33]{Jan97}) implies
  \begin{equation*}
    Z_n^{:k:}(t, x) Z_m^{:l:}(t, x)
    = \sum_{r=0}^{\min\{k,l\}} \rconst_n^r Z_n^{:k-r:}(t, x) \wickproduct Z_m^{:l-r:}(t, x),
  \end{equation*}
  where $\rconst_n \defby \int_{\R \times \torus} \kernel_n(s, y)^2 ds dy \sim \log n$.

  We next compute
  \begin{equation*}
    A \defby \expect[ \inp{Z_n^{:k:}(t_1) \wickproduct Z_m^{:l:}(t_1)}{\phi}
    \inp{Z_n^{:k:}(t_2) \wickproduct Z_m^{:l:}(t_2)}{\phi}].
  \end{equation*}
  By It\^o isometry, $A$ equals
  \begin{multline*}
    \int_{(\R \times \torus)^{k+l}} \left( \int_{\torus} \prod_{j=1}^k \kernel_n(t_1-s_j, x_1-y_j) \prod_{j=k+1}^{k+l}
    \kernel_m(t_1-s_j, x_1-y_j) \phi(x_1) dx_1 \right) \\
    \times \left( \int_{\torus} \prod_{j=1}^k \kernel_n(t_2-s_j, x_2-y_j) \prod_{j=k+1}^{k+l} \kernel_m(t_2-s_j, x_2-y_j) \phi(x_2) dx_2
    \right) \prod_{j=1}^{k+l} ds_j dy_j.
  \end{multline*}
  By changing the order of integrations, we obtain
  \begin{multline*}
    A = \int_{\torus \times \torus} dx_1 dx_2 \phi(x_1) \phi(x_2)
    \left( \int_{-\infty}^{\min\{t_1, t_2\}} \kernel_n(t_1+t_2-2s, x_1-x_2) ds \right)^k \\
    \times \left( \int_{-\infty}^{\min\{t_1, t_2\}} \kernel_m(t_1+t_2-2s, x_1-x_2) ds \right)^l,
  \end{multline*}
  which, applying Plancherel theorem, equals
  \begin{equation*}
    \sum_{\substack{\abs{p_1}, \ldots, \abs{p_k} \leq n \\ \abs{p_{k+1}}, \ldots, \abs{p_{k+l}} \leq m}}
    \abs{\hat{\phi}(p_1 + \cdots p_{k+l})}^2  \prod_{j=1}^{k+l} \frac{e^{-I_{p_j} \abs{t_1-t_2}}}{2 I_{p_j}},
  \end{equation*}
  where $I_p \defby 1 + 4 \pi^2 \abs{p}^2$.
  By change of variables, we conclude
  \begin{equation*}
    A = \sum_{p \in \Z^2} \abs{\hat{\phi}(p)}^2 \left\{K_n(t_1-t_2)^{\conv k} \conv K_m(t_1 - t_2)^{\conv l}
    \right\}(p),
  \end{equation*}
  where $K(t, p) \defby (2I_p)^{-1} e^{-I_p \abs{t}}$ and $K_n(t) \defby (K(t))_n$.

  Similar computations yield
  \begin{align*}
    \expect[ \inp{Z_n^{:k+l:}(t_1)}{\phi} \inp{Z_n^{:k+l:}(t_2)}{\phi}]
    &= \expect[ \inp{Z_n^{:k+l:}(t_1)}{\phi} \inp{Z_n^{:k:}(t_2) \wickproduct Z_m^{:l:}(t_2)}{\phi} ] \\
    &= \expect[ \inp{Z_n^{:k:}(t_1) \wickproduct Z_m^{:l:}(t_1)}{\phi} \inp{Z_n^{:k+l:}(t_2)}{\phi}  ] \\
    &= \sum_{p \in \Z^2} \abs{\hat{\phi}(p)}^2 K_n(t_1-t_2)^{\conv (k+l)}(p).
  \end{align*}

 Consequently, we have
 \begin{align*}
   \MoveEqLeft[3] \expect[\abs{ \inp{Z_n^{:k:}(t) \wickproduct (Z_m^{:l:}(t) - Z_n^{:l:}(t))}{\eta_k(x - \cdot)}}^2 ] \\
   &= \sum_{p \in \Z^2} \abs{\chi_k(p)}^2 \left\{
   K_n(0)^{\conv k} \conv K_m(0)^{\conv l}(p) - K_n^{\conv (k+l)}(p) \right\} \\
   &\lesssim \frac{1}{(1+n^2)^{\epsilon}}\sum_{p \in \Z^2} \frac{\abs{\chi_k(p)}^2}{(1 + \abs{p}^2)^{\beta}}
   \lesssim \frac{2^{-2k(1-\beta)}}{(1 + n^2)^{\epsilon}},
 \end{align*}
 where $\epsilon, \beta \in (0, 1)$ with $\epsilon + \beta < 1$.
 We applied Lemma \ref{lem:kernel_estimate} to derive the inequality above.
 Similarly,
 \begin{align*}
   \MoveEqLeft \expect[ \abs{ \inp{Z_n^{:k:}(t_1) \wickproduct (Z_m^{:l:}(t_1) - Z_n^{:l:}(t_1)) -
   Z_n^{:k:}(t_2) \wickproduct (Z_m^{:l:}(t_2) - Z_n^{:l:}(t_2))}{\eta_k(x-\cdot)}}^2 ] \\
   \begin{split}
     {}&= 2\left[ \sum_{p \in \Z^2} \abs{\chi_k(p)}^2 \Big\{ \big(K_n(0)^{\conv k} \conv K_m(0)^{\conv l}\big) (p) -
     K_n(0)^{\conv(k+l)}(p) \Big\} \right. \\
     &\,\, \left. - \sum_{p \in \Z^2}
     \abs{\chi_k(p)}^2  \Big\{ \big(K_n(t_1-t_2)^{\conv k} \conv K_m(t_1-t_2)^{\conv l}\big) (p) -
     K_n(t_1-t_2)^{\conv(k+l)}(p) \Big\} \right]
   \end{split} \\
   {}& \lesssim \abs{t_1-t_2}^{\epsilon} \sum_{p \in \Z^2} \frac{\abs{\chi_k(p)}^2}{(1 + \abs{p}^2)^{\beta}}
   \lesssim \abs{t_1-t_2}^{\epsilon} 2^{-2k(1-\beta)}.
 \end{align*}
 We again applied Lemma \ref{lem:kernel_estimate} by noting
 \begin{equation*}
   \abs{K_n(t,m) - K_n(0,m)} \lesssim_{\gamma} t^{\gamma} (1 + \abs{m}^2)^{\gamma - 1}
 \end{equation*}
 for $\gamma \in (0, 1)$.
 Therefore, we finally obtain
 \begin{multline*}
   \expect[ \abs{ \inp{Z_n^{:k:}(t_1)(Z_m^{:l:}(t_1) - Z_n^{:l:}(t_1)) -
   Z_n^{:k:}(t_2)(Z_m^{:l:}(t_2) - Z_n^{:l:}(t_2))}{\eta_k(x-\cdot)}}^2 ] \\
   \lesssim \abs{t_1 - t_2}^{\frac{\epsilon}{2}} (1 + n^2)^{-\frac{\epsilon}{2}} 2^{-2k(1-\beta)}.
 \end{multline*}
 Applying a Besov space version of Kolmogorov continuity theorem, we complete the proof.
\end{proof}
\begin{corollary}\label{cor:fin_dim_approxim_of_she}
  Let $T \in (0, \infty)$. For every $k \in \N$, we have
  \begin{equation*}
    \lim_{n \to \infty} \expect[ \sup_{0 \leq t \leq T} \norm{Z^{:k:}(t) - Z^{:k:}_n(t)}_{\contisp^{-\alpha}}^2 ] = 0.
  \end{equation*}
\end{corollary}
\begin{proof}
  This easily follows form Lemma \ref{lem:fin_dim_approxim_of_she}.
\end{proof}
\section*{Acknowledgements}
The author would like to thank  Prof. Yuzuru Inahama and Dr. Masato Hoshino
for helpful comments.
\printbibliography
\end{document}